\documentclass[a4paper,12pt]{article}
\usepackage[top=2.5cm,bottom=2.5cm,left=2.5cm,right=2.5cm]{geometry}
\usepackage{amsmath,amssymb,amsthm,graphics,latexsym,amsfonts}
\usepackage{makecell}
\usepackage{caption}
\usepackage{booktabs}
\usepackage{graphicx}
\usepackage{bm,bbm}
\usepackage{mathrsfs}
\usepackage{tabularx}
\usepackage{cite}
\usepackage{indentfirst}

\newtheorem{thm}{Theorem}[section]
\newtheorem{lem}[thm]{Lemma}
\newtheorem{cor}[thm]{Corollary}

\newtheorem{eg}[thm]{Example}

\newtheorem{alg}[thm]{Algorithm}

\makeatletter
\newcommand{\Rmnum}[1]{\expandafter\@slowromancap\romannumeral #1@} 
\makeatother
\allowdisplaybreaks[3] 

\begin{document}

\title{{Matching forcing polynomials of constructable hexagonal systems}}
\author{Shuang Zhao\thanks{E-mail: zhaosh2018@126.com}\\
{\small School of Information Engineering, Lanzhou University of Finance and Economics, }\\
{\small Lanzhou, Gansu 730020, P. R. China}}
\date{}
\maketitle

\begin{abstract}
    In this paper, we derive recurrence relations of forcing polynomials for monotonic CHS and the other is CHS with one turning.
\end{abstract}

\section{Introduction}
    Let $G$ be a graph with vertex set $V(G)$ and edge set $E(G)$. A \emph{perfect matching} $M$ of $G$ is a set of disjoint edges that covers all vertices of $G$. A cycle of $G$ is called \emph{$M$-alternating} if its edges appear alternately in $M$ and $E(G)\setminus M$. A \emph{forcing set} $S$ of $M$ is a subset of $M$ such that $S$ is contained in no other perfect matchings of $G$. The \emph{forcing number} of $M$, denoted by $f(G,M)$, is the smallest cardinality over all forcing sets of $M$. Klein and Randi\'c \cite{early,RK} proposed the innate degree of freedom of a Kekul\'e structure, nowadays it is called the forcing number of a perfect matching by Harary et al. \cite{original}.

    An edge of a graph $G$ is called \emph{forcing} if it belongs to precisely one perfect matching of $G$. The \emph{maximum} (resp. \emph{minimum}) \emph{forcing number} of $G$ is the maximum (resp. minimum) value of $f(G,M)$ over all perfect matchings $M$ of $G$, denoted by $F(G)$ (resp. $f(G)$). The \emph{forcing spectrum} of $G$ is the set of forcing numbers of all perfect matchings of $G$. The author, Zhang and Lin \cite{zhao} introduced the \emph{forcing polynomial} of a graph $G$ as
    \begin{align}
        \label{equ1}
        F(G,x)=\sum_{M\in\mathcal{M}(G)}{{x}^{f(G,M)}}=\sum_{i=f(G)}^{F(G)}{\omega(G,i){x}^{i}},
    \end{align}
    where $\mathcal{M}(G)$ denotes the set of all perfect matchings of $G$, and $\omega(G,i)$ denotes the number of perfect matchings of $G$ with forcing number $i$.
%

    A \emph{hexagonal system} (or \emph{benzenoid system}) $H$ is a 2-connected finite plane graph such that every interior face is surrounded by a regular hexagon. Xu et al. \cite{Xu} showed that the maximum forcing number of $H$ equals the \emph{Clar number} (i.e. the maximum number over disjoint alternating hexagons with respect to a perfect matching) of $H$, which can measure the stability of benzenoid hydrocarbons.

    A hexagonal system is said to be a \emph{constructable hexagonal system}, or briefly \emph{CHS} \cite{dingli2}, if it can be dissected by parallel lines $L_1,L_2,\ldots,L_m$ that are perpendicular to some of its edges, such that it decomposes into $m+1$ paths $P_1,P_2,\ldots,P_{m+1}$, the first one $P_1$ and the last one $P_{m+1}$ must be of even length, and all the other paths are of odd length. We can see that all the hexagons which intersect $L_i$ form a linear hexagonal chain, called the $i$th \emph{row} of CHS for $i=1,2,\ldots,m$. For convenience, we always place CHS satisfying that each $L_i$ is horizontal, see Fig. \ref{mono-chs}(a). Zhang and Li \cite{dingli2} proved that a CHS has a perfect matching, and every perfect matching of a CHS contains precisely one vertical edge in each row.
    \begin{figure}[htbp]
        \centering
        \includegraphics[height=2.2in]{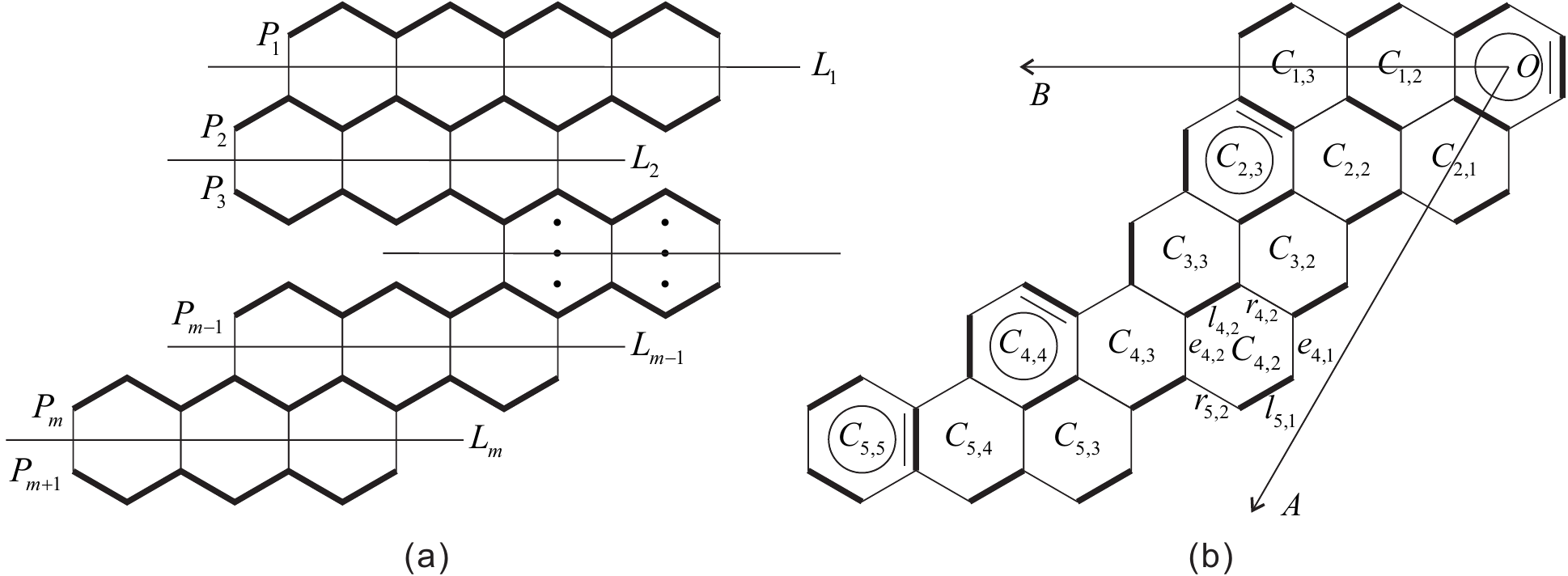}
        \caption{(a) A CHS with $m$ rows, and (b) $CHS(3,3,3,4,5;1,1,2,2,3)$ with perfect matching $(0,3,3,4,4)$.}
        \label{mono-chs}
    \end{figure}
\section{Some preliminaries}
    \begin{thm}{\em\cite{1,15}}
        \label{forc-equi-defi}
        Let $G$ be a graph with a perfect matching $M.$ A subset $S\subseteq M$ is a forcing set of $M$ if and only if each $M$-alternating cycle of $G$ contains at least one edge of $S.$
    \end{thm}

    From the above theorem we can see that the forcing number $f(G,M)$ is bounded below by $C(G,M)$, the maximum number of disjoint $M$-alternating cycles. Furthermore, we have the following result.

    \begin{thm}{\em\cite{squaregrids}}
        \label{forc-min-max}
        Let $G$ be a planar bipartite graph. Then for each perfect matching $M$ of $G,$ we have $f(G,M)=C(G,M).$
    \end{thm}

    Obviously, the above theorem holds for hexagonal systems. For $S\subseteq E(G)$, let $G-V(S)$ denote the subgraph obtained from $G$ by deleting the ends of edges that belong to $S$. An edge $e$ is said to be \emph{defined} by $V(S)$ if it contains in all perfect matchings of $G-V(S)$. Let $G\circleddash V(S)$ denote the subgraph obtained from $G-V(S)$ by deleting the ends of edges that defined by $V(S)$.

    \begin{lem}
        \label{forc-pend-edge-judg}
        Let $G$ be a graph with a perfect matching $M.$ A subset $S\subseteq M$ is a forcing set of $M$ if and only if $G\circleddash V(S)$ is empty$.$
    \end{lem}
    \begin{proof}
        Given a subset $S\subseteq M$. $S$ is a forcing set of $M$ is equivalent to $G-V(S)$ has a unique perfect matching $M\setminus S$, and is equivalent to $G\circleddash V(S)$ is empty.
    \end{proof}

    \begin{lem}
        \label{mini-forc}
        Let $G$ be a graph with a perfect matching $M,$ and $\mathcal{C}$ be a set of disjoint $M$-alternating cycles of $G.$ Given a subset $S\subseteq M,$ which consists of precisely one edge from each cycle in $\mathcal{C}.$ If $V(S)$ defines all the other edges in $M\cap E(\mathcal{C}),$ then $S$ is contained in some minimum forcing set of $M.$
    \end{lem}
    \begin{proof}
        Let $G'=G\circleddash V(S)$, $G''=G-V(G')$. Then $M\cap E(G')$ is a perfect matching of $G'$, $M\cap E(G'')$ is a perfect matching of $G''$, and $S\subseteq M\cap E(G'')$. Since there is $|S|$ number of disjoint $M$-alternating cycles in $G''$ and $G''\circleddash V(S)$ is empty, $S$ is a minimum forcing set of $M\cap E(G'')$ in $G''$ by Theorem \ref{forc-equi-defi} and Lemma \ref{forc-pend-edge-judg}. Let $S'$ be a minimum forcing set of $M\cap E(G')$ in $G'$. Since $G\circleddash V(S\cup S')=G\circleddash V(S)\circleddash V(S')=G'\circleddash V(S')$ is empty, we know that $S\cup S'$ is a forcing set for $M$ in $G$ by Lemma \ref{forc-pend-edge-judg}. Suppose $S_0$ is another forcing set for $M$ in $G$ such that $|S_0|<|S\cup S'|$. Then either $|S_0\cap E(G'')|<|S|$, or $|S_0\cap E(G')|<|S'|$. It follows that either $S_0\cap E(G'')$ is not a forcing set for $M\cap E(G'')$ in $G''$, or $S_0\cap E(G')$ is not a forcing set for $M\cap E(G')$ in $G'$. This implies that there is an $M$-alternating cycle in $G''$ or $G'$ containing no edges in $S_0$, which is a contradiction to Theorem \ref{forc-equi-defi}.
    \end{proof}

    \begin{lem}
        \label{forc-subs-calc}
        Let $G$ be a graph with a perfect matching $M$ and a minimum forcing set $S$ of $M.$ If $S'\subseteq S,$ then
        \begin{align*}
            f(G,M)=f(G- V(S'),M\setminus S')+|S'|=f(G\circleddash V(S'),M\cap E(G\circleddash V(S')))+|S'|.
        \end{align*}
    \end{lem}
    \begin{proof}
        The second equation is obvious since every edge that is defined by $V(S')$ must belong to $M\setminus S'$ in $G-V(S')$. We now consider the first equation. Since $S$ is a forcing set of $M$ in $G$, $S\setminus S'$ is a forcing set of $M\setminus S'$ in $G-V(S')$. On the other hand, if $S_0$ is a forcing set of $M\setminus S'$ in $G-V(S')$, then $S'\cup S_0$ is a forcing set of $M$ in $G$. It follows that
        \begin{align*}
            |S|=f(G,M)\le f(G-V(S'),M\setminus S')+|S'|\le |S\setminus S'|+|S'|,
        \end{align*}
        which implies the first equation.
    \end{proof}

\section{Monotonic CHS}
    A CHS is called \emph{left-monotonic} (resp. \emph{right-monotonic}) if the leftmost hexagon in each row is located on the left (resp. right) of the leftmost hexagon in the row immediately above. Left-monotonic CHS and right-monotonic CHS both are called \emph{monotonic CHS}. Since inverting a right-monotonic CHS upside down derives a left-monotonic CHS, we only talk about left one in the following, see Fig. \ref{mono-chs}(b). In order to label each hexagon, we suppose the side of every hexagon has length $\frac{\sqrt{3}}{3}$, which implies that the length between the center of two adjacent hexagon is 1. Denote the hexagon in the top right corner by $C_{1,1}$, and its center by $O$. From $O$ draw two rays $OA$ and $OB$ perpendicular to the bottom left oblique edge and the left vertical edge of $C_{1,1}$, respectively. Denote a hexagon by $C_{i,j}$, if two lines can be drawn through its center $W$ such that one is parallel to axis $OB$ and intersects axis $OA$ at the point $W_A$ and the other is parallel to axis $OA$ and intersects $OB$ at the point $W_B$, and the length of $OW_A$ is $i-1$ and the length of $OW_B$ is $j-1$, see Fig. \ref{mono-chs}(b).

    If a monotonic CHS has $m(\geqslant 1)$ rows, and the leftmost and rightmost hexagons in $i$th row are $C_{i,k_i}$ and $C_{i,h_i}$ respectively for $i=1,2,\ldots,m$ ($k_{i+1}\geqslant k_{i}$, $h_{i+1}\geqslant h_{i}$, $k_j\geqslant h_j$ for $i=1,2,\ldots,m-1$, $j=1,2,\ldots,m$), then we denote it by $CHS(k_1,k_2,\ldots,k_m;h_1,h_2,\ldots,h_m)$, or briefly $CHS(\{k_s;h_s\}_{s=1}^{m})$. Furthermore, for $i=1,2,\ldots,m$ and $j=h_i,h_{i}+1,\ldots,k_i$, denote the left vertical, top left oblique, top right oblique, right vertical, bottom right oblique, and bottom left oblique edges of $C_{i,j}$ by $e_{i,j}$, $l_{i,j}$, $r_{i,j}$, $e_{i,j-1}$, $l_{i+1,j-1}$, and $r_{i+1,j}$, respectively, see Fig. \ref{mono-chs}(b).

    In particular, $CHS(\{k_s;1\}_{s=1}^{m})$ with $1\leqslant k_1\leqslant k_2\leqslant\cdots\leqslant k_m$ is a truncated parallelogram (see Fig. \ref{eg-mono-chs}(a)), $CHS(k;1)$ is a linear hexagonal chain (see Fig. \ref{eg-mono-chs}(b)), $CHS(1,2,\ldots,k,k;1,1,2,\ldots,k)$ is a zigzag hexagonal chain with even number of hexagons (see Fig. \ref{eg-mono-chs}(c)), $CHS(2,3,\ldots,k,k;1,2,\ldots,k)$ is a zigzag hexagonal chain with odd number of hexagons (see Fig. \ref{eg-mono-chs}(d)), and $CHS(\{k;1\}_{s=1}^{m})$ is a benzenoid parallelogram (see Fig. \ref{eg-mono-chs}(e)).

    \begin{figure}[htbp]
        \centering
        \includegraphics[height=1.15in]{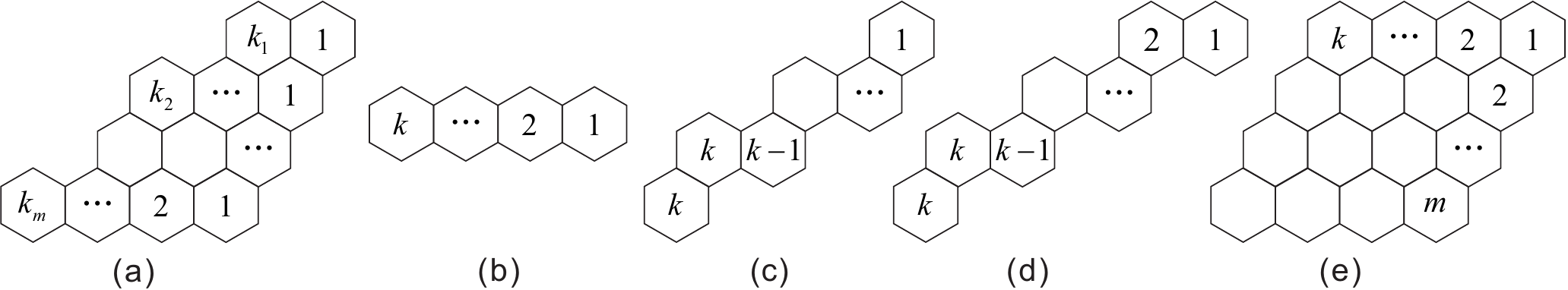}
        \caption{(a) $CHS(\{k_s;1\}_{s=1}^{m})$, (b) $CHS(k;1)$, (c) $CHS(1,2,\ldots,k,k;1,1,2,\ldots,k)$, (d) $CHS(2,3,\ldots,k,k;1,2,\ldots,k)$, and (e) $CHS(\{k;1\}_{s=1}^{m})$.}
        \label{eg-mono-chs}
    \end{figure}

    In order to derive forcing polynomial of monotonic CHS's, we do some preliminaries first.

    \begin{lem}{\em\cite{dingli2}}
        \label{ding}
        There is a bijection $g$ between all perfect matchings $M$ of $CHS(\{k_s;$\\$h_s\}_{s=1}^{m})$ and all non-decreasing sequences $(a_1,a_2,\ldots,a_m)$ with $a_i\in \{h_i-1,h_i,\ldots,k_i\}$ for $i=1,2,\ldots,m,$ such that $g(M)=(a_1,a_2,\ldots,a_m)$ if $e_{i,a_i}\in M$ for $i=1,2,\ldots,m.$
    \end{lem}

    From the above lemma, we can use a sequence to express each perfect matching of a monotonic CHS for convenience. As an example, the perfect matching of $CHS(3,3,3,4,5;$\\$1,1,2,2,3)$ illustrated with a set of bold lines in Fig. \ref{mono-chs}(b) can be expressed by (0,3,3,4,4).

    Now we can give a recurrence relation of the forcing polynomial for $CHS(\{k_s;h_s\}_{s=1}^{m})$. For convenience, from now on we define $CHS(\{b_s;d_s\}_{s=1}^{i})=CHS(\{b_s;d_s\}_{s=1}^{i-1})$ if $b_i<d_i$ and $i\geqslant1$, and $CHS(\{b_s;d_s\}_{s=1}^{j})$ as an empty graph if $j=0$. And we make a convention that
    \begin{align*}
        \mathcal{M}(\{k_s;h_s\}_{s=1}^{m})&=\mathcal{M}(CHS(\{k_s;h_s\}_{s=1}^{m})),\\
        F(\{k_s;h_s\}_{s=1}^{m})&=F(CHS(\{k_s;h_s\}_{s=1}^{m}),x).
    \end{align*}
    According to which vertical edge in the last row that belongs to perfect matching, we divide $\mathcal{M}(\{k_s;h_s\}_{s=1}^{m})$ in $k_m-h_m+2$ subsets:
    \begin{align*}
        \mathcal{M}_i(\{k_s;h_s\}_{s=1}^{m})=\{M\in \mathcal{M}(\{k_s;h_s\}_{s=1}^{m}): e_{m,i}\in M\}
    \end{align*}
    for $i=h_m-1,h_m,\ldots,k_m$. By Eq. (\ref{equ1}), we have
    \begin{align}
        \label{mono0}
        F(\{k_s;h_s\}_{s=1}^{m})=&\sum_{i=h_m-1}^{k_m}\sum \limits_{ M\in {\mathcal{M}_i(\{k_s;h_s\}_{s=1}^{m})}}{ x^{ f(CHS(\{k_s;h_s\}_{s=1}^{m}),M) } }\nonumber\\
        :=&\sum_{i=h_m-1}^{k_m}F_i(\{k_s;h_s\}_{s=1}^{m}).
    \end{align}

    \begin{thm}
        \label{mono-chs-forc}
        The forcing polynomial of $CHS(\{k_s;h_s\}_{s=1}^{m})$ $(k_{i+1}\geqslant k_{i},$ $h_{i+1}\geqslant h_{i},$ $k_j\geqslant h_j$ for $i=1,2,\ldots,m-1,$ $j=1,2,\ldots,m,~m\geqslant 1)$ has the following recurrence relation$:$
        \begin{align*}
            F(\{k_s;h_s\}_{s=1}^{m})=&\sum_{i=h_m-1}^{k_m-1}F(\{\min\{k_s,i\};h_s\}_{s=1}^{m-1})x+\sum_{j=p}^mF(\{\min\{k_s,k_m-1\};h_s\}_{s=1}^{j-1})x
        \end{align*}
        where $p=\min\{s:k_s=k_m\},$ $F(\{b_s;d_s\}_{s=1}^{i})=F(\{b_s;d_s\}_{s=1}^{i-1})$ if $b_i<d_i$ and $i\geqslant 1,$ and $F(\{b_s;d_s\}_{s=1}^{j})=1$ if $j=0.$
    \end{thm}
    \begin{proof}
        Given $i\in \{h_m-1,h_m,\ldots,k_m-1\}$ and $M\in \mathcal{M}_i(\{k_s;h_s\}_{s=1}^{m})$. On the one hand $e_{m,i}$ belongs to $M$-alternating hexagon $C_{m,i+1}$, and on the other hand in $CHS(\{k_s;h_s\}_{s=1}^{m})-V(e_{m,i})$ the lowermost (resp. uppermost) vertex of $C_{m,i+1}$ must be covered by the edge $r_{m+1,i+1}$ (resp. $l_{m,i+1}$) in $M$ by Lemma \ref{ding}. It follows that $e_{m,i}$ belongs to some minimum forcing set of $M$ by Lemma \ref{mini-forc}. Furthermore, it is observed from Fig. \ref{mono-chs-proof}(a) that
        \begin{align*}
            CHS(\{k_s;h_s\}_{s=1}^{m})\circleddash V(e_{m,i})=CHS(\{\min \{k_s,i\};h_s\}_{s=1}^{m-1}).
        \end{align*}
        By Lemma \ref{forc-subs-calc}, we have
        \begin{align}
            \label{mono1}
            F_i(\{k_s;h_s\}_{s=1}^{m})=&\sum _{ M\in {\mathcal{M}_i(\{k_s;h_s\}_{s=1}^{m})}}{ x^{ f(CHS(\{\min \{k_s,i\};h_s\}_{s=1}^{m-1}),M\cap E(CHS(\{\min \{k_s,i\};h_s\}_{s=1}^{m-1})))+1}}\nonumber \\
            =&\sum _{ M\in {\mathcal{M}(\{\min\{k_s,i\};h_s\}_{s=1}^{m-1})}}{x^{f(CHS(\{\min \{k_s,i\};h_s\}_{s=1}^{m-1}),M)}}\cdot x \nonumber\\
            =&F(\{\min\{k_s,i\};h_s\}_{s=1}^{m-1})x.
        \end{align}

        \begin{figure}[htbp]
            \centering
            \includegraphics[height=1.75in]{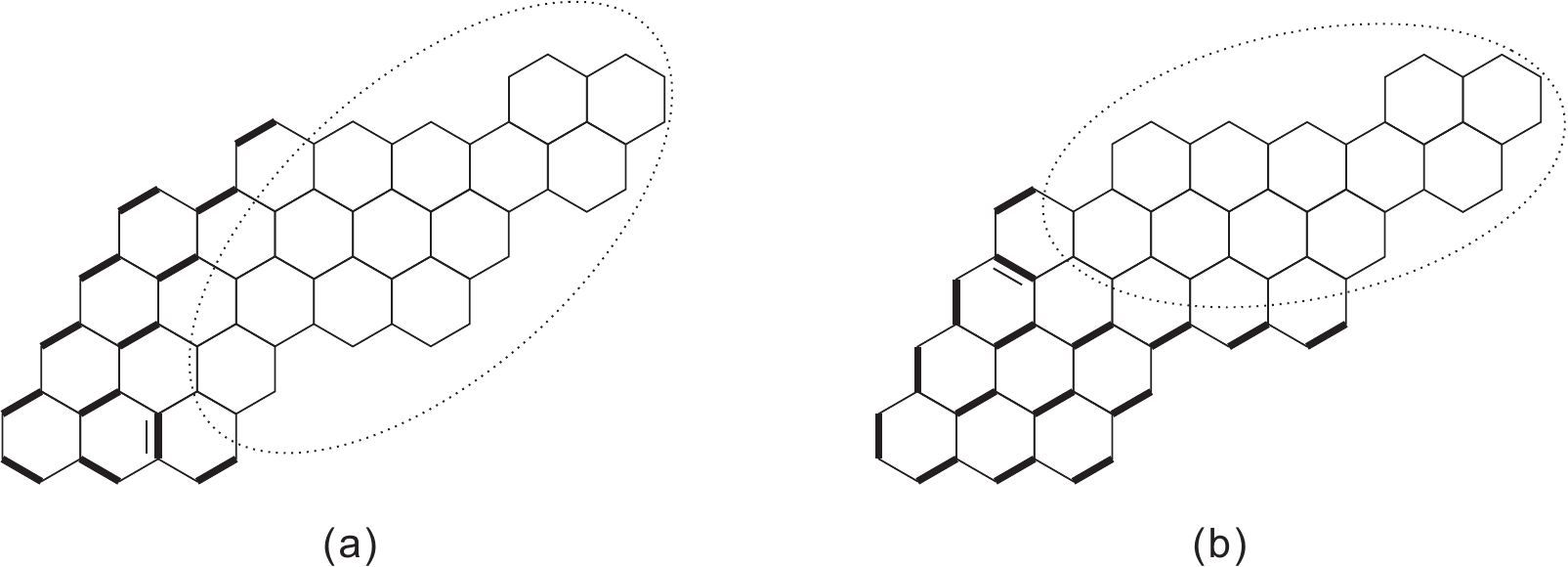}
            \caption{(a) $CHS(\{k_s;h_s\}_{s=1}^{m})\circleddash V(e_{m,i})$, and (b) $CHS(\{k_s;h_s\}_{s=1}^{m})\circleddash V(r_{j,k_m})$.}
            \label{mono-chs-proof}
        \end{figure}

        Given $M\in \mathcal{M}_{k_m}(\{k_s;h_s\}_{s=1}^{m})$. Let $j=\min\{j:e_{j,k_m}\in M\}$. Then $p\leqslant j\leqslant m$ and $r_{j,k_m}\in M$. On the one hand $r_{j,k_m}$ belongs to $M$-alternating hexagon $C_{j,k_m}$, and on the other hand in $CHS(\{k_s;h_s\}_{s=1}^{m})-V(r_{j,k_m})$ the leftmost two vertices of $C_{j,k_m}$ must be matched with each other in $M$, and the lowermost vertex of $C_{j,k_m}$ must be covered by the edge $l_{j+1,k_m-1}$ in $M$ by Lemma \ref{ding}. It follows that $r_{j,k_m}$ belongs to some minimum forcing set of $M$ by Lemma \ref{mini-forc}. Furthermore, it is observed from Fig. \ref{mono-chs-proof}(b) that
        \begin{align*}
            CHS(\{k_s;h_s\}_{s=1}^{m})\circleddash V(r_{j,k_m})=CHS(\{\min \{k_s,k_m-1\};h_s\}_{s=1}^{j-1}).
        \end{align*}
        By Lemma \ref{forc-subs-calc}, we have
        \begin{align}
            \label{mono2}
            &F_{k_m}(\{k_s;h_s\}_{s=1}^{m})\nonumber\\
            =&\sum_{j=p}^m\sum_{\begin{subarray}{c}M\in {\mathcal{M}_{k_m}(\{k_s;h_s\}_{s=1}^{m})} \\ j=\min\{j:e_{j,k_m}\in M\}\end{subarray}}{ x^{ f(CHS(\{\min \{k_s,k_m-1\};h_s\}_{s=1}^{j-1}),M\cap E(CHS(\{\min \{k_s,k_m-1\};h_s\}_{s=1}^{j-1})))+1}}\nonumber \\
            =&\sum_{j=p}^m\sum _{ M\in {\mathcal{M}(\{\min \{k_s,k_m-1\};h_s\}_{s=1}^{j-1})}}{{x}^{f(CHS(\{\min \{{{k}_{s}},{{k}_{m}}-1\};{{h}_{s}}\}_{s=1}^{j-1}),M)}}\cdot x\nonumber\\
            =&\sum_{j=p}^mF(\{\min \{k_s,k_m-1\};h_s\}_{s=1}^{j-1})x.
        \end{align}
        Substituting Eqs. (\ref{mono1},\ref{mono2}) into Eq. (\ref{mono0}), we immediately obtain the theorem.
    \end{proof}

    Li \cite{antiedge} characterized hexagonal systems with anti-forcing edges, which are truncated parallelograms. From the above theorem, we can derive the forcing polynomial.

    \begin{cor}
        \label{trum}
        The forcing polynomial of truncated parallelogram $CHS(\{k_s;1\}_{s=1}^{m})$ $(1\leqslant k_1\leqslant k_2\leqslant \cdots \leqslant k_m,$ $m\geqslant 1)$ (see Fig. \ref{eg-mono-chs}(a)) has the following recurrence relation$:$
        \begin{align*}
            F(\{k_s;1\}_{s=1}^{m})=&\sum_{i=0}^{k_m-1}F(\{\min\{k_s,i\};1\}_{s=1}^{m-1})x+\sum_{j=\min\{s:k_s=k_m\}}^mF(\{\min\{k_s,k_m-1\};1\}_{s=1}^{j-1})x
        \end{align*}
        where $F(\{b_s;1\}_{s=1}^{i})=F(\{b_s;1\}_{s=1}^{i-1})$ if $b_i=0$ and $i\geqslant 1,$ and $F(\{b_s;1\}_{s=1}^{j})=1$ if $j=0.$
    \end{cor}

    Zhang and Deng \cite{deng} obtained the continuity of forcing spectrum for truncated parallelograms by using Z-transform graph, and here we show the result by the degrees of forcing polynomial as follows.

    \begin{cor}{\em\cite{deng}}
        The forcing spectrum of truncated parallelogram $CHS(\{k_s;1\}_{s=1}^{m})$ $(1\leqslant k_1\leqslant k_2\leqslant \cdots \leqslant k_m,$ $m\geqslant 1)$ is an integer interval from 1$.$
    \end{cor}
    \begin{proof}
        We proceed by induction on the number of rows $m$. For initial case of $m=1$, obviously the forcing spectrum of $CHS(k_1;1)$ is \{1\}. Suppose that the result holds for the cases of less than $m(\geqslant 2)$. Now we consider the case of $m$. By Corollary \ref{trum}, we have
        \begin{align*}
            F(\{k_s;1\}_{s=1}^{m})=&\sum_{i=1}^{k_m-1}F(\{\min\{k_s,i\};1\}_{s=1}^{m-1})x\\
            &+\sum_{j=\min\{s:k_s=k_m\}}^mF(\{\min\{k_s,k_m-1\};1\}_{s=1}^{j-1})x+x.
        \end{align*}
        By inductive hypothesis, we know that the degrees of nonzero terms in each $F(\{\min\{k_s,i\};$\\$1\}_{s=1}^{m-1})x$ or $F(\{\min\{k_s,k_m-1\};1\}_{s=1}^{j-1})x$ in the above equation form an integer interval from 1 or 2; and the last term has degree 1, which implies the forcing spectrum is an integer interval from 1.
    \end{proof}

    We now give some forcing polynomials of particular monotonic CHS's.

    \begin{eg}{\em\cite{zhao}}
        \label{linear}
        The forcing polynomial of linear hexagonal chain $CHS(k;1)$ $(k\geqslant 1)$ (see Fig. \ref{eg-mono-chs}(b)) is
        \begin{align*}
            F(CHS(k;1),x)=\sum_{i=0}^{k-1}x+x=(k+1)x.
        \end{align*}
    \end{eg}

    \begin{eg}{\em\cite{zhao}}
        The forcing polynomial of zigzag hexagonal chain $Z_n$ with $n(\geqslant 3)$ hexagons (see Fig. \ref{eg-mono-chs}(c,d)) has the following recurrence relation$:$
        \begin{align*}
            F(Z_n,x)=2F(Z_{n-2},x)x+F(Z_{n-3},x)x
        \end{align*}
        with initial conditions $F(Z_0,x)=1,$ $F(Z_1,x)=2x,$ $F(Z_2,x)=3x.$
    \end{eg}
    \begin{proof}
        From Theorem \ref{mono-chs-forc}, for $n=2k$ and $k\geqslant 2$ we have
        \begin{align*}
        F(Z_n,x)=&F(CHS(1,2,\ldots,k,k;1,1,2,\ldots,k),x)\\
        =&2F(CHS(1,2,\ldots,k-1,k-1;1,1,2,\ldots,k-1),x)x\\
        &+F(CHS(1,2,\ldots,k-1;1,1,2,\ldots,k-2),x)x\\
        =&2F(Z_{n-2},x)x+F(Z_{n-3},x)x.
        \end{align*}
        For $n=2k-1$ and $k\geqslant 2$ we have
        \begin{align*}
            F(Z_n,x)=&F(CHS(2,3,\ldots,k,k;1,2,\ldots,k),x)\\
            =&2F(CHS(2,3,\ldots,k-1,k-1;1,2,\ldots,k-1),x)x\\
            &+F(CHS(2,3,\ldots,k-1;1,2,\ldots,k-2),x)x\\
            =&2F(Z_{n-2},x)x+F(Z_{n-3},x)x.
        \end{align*}
        The initial conditions are easy to verify.
    \end{proof}

    From the above recurrence relation, we can derive the explicit form of forcing polynomial for zigzag hexagonal chains, which can be seen in Ref. \cite{zhao}. In fact, from Theorem \ref{mono-chs-forc} we can derive recurrence relation of forcing polynomial for an arbitrary hexagonal chain, which coincides with that in Ref. \cite{zhao}.

    \begin{eg}{\em\cite{zhao1}}
        The forcing polynomial of benzenoid parallelogram $M(k,m)$ with $k(\geqslant 1)$ columns and $m(\geqslant 1)$ rows (see Fig. \ref{eg-mono-chs}(e)) has the following recurrence relation$:$
        \begin{align*}
            F(M(k,m),x)=&F(\{k;1\}_{s=1}^{m})=\sum_{i=0}^{k-1}F(\{i;1\}_{s=1}^{m-1})x+\sum_{j=1}^{m}F(\{k-1;1\}_{s=1}^{j-1})x\\
            =&\sum_{i=0}^{k-1}F(M(i,m-1),x)x+\sum_{j=0}^{m-1}F(M(k-1,j),x)x
        \end{align*}
        with initial conditions $F(M(0,n),x)=F(M(n,0),x)=1$ for $n\geqslant 0.$
    \end{eg}

    From the above recurrence relation, we can also derive the explicit form of forcing polynomial for benzenoid parallelogram, which can be seen in Ref. \cite{zhao1}. In the end of this section, we give an algorithm to find a minimum forcing set $S_M$ of perfect matching $M$ for $CHS(\{k_s;h_s\}_{s=1}^{m})$ $(k_{i+1}\geqslant k_{i},$ $h_{i+1}\geqslant h_{i},$ $k_j\geqslant h_j$ for $i=1,2,\ldots,m-1,$ $j=1,2,\ldots,m,~m\geqslant 1)$. The proof is similar to that of Theorem \ref{mono-chs-forc}, and we omit it here.

    \begin{alg}
        \label{alg1} ~\\
        \em\textbf{Input:} $CHS(k_1,k_2,\ldots,k_m;h_1,h_2,\ldots,h_m)$ with perfect matching $M=(a_1,a_2,\ldots,a_m)$.\\
        \textbf{Output:} A minimum forcing set $S_M$ of $M$.\\
        (1) Let $t\leftarrow m$; $j\leftarrow m$; $S\leftarrow \emptyset$;\\
        \hspace*{1cm}\textbf{while} $t\geqslant 1$ \textbf{do}\\
        \hspace*{2cm}$l_t\leftarrow k_t$, $t\leftarrow t-1$.\\
        (2) \textbf{While} $j\geqslant 1$ \textbf{do}\\
        \hspace*{1cm}\textbf{if} $l_j\geqslant h_j $ \textbf{then}\\
        \hspace*{2cm}\textbf{if} $a_j=l_j$ \textbf{then} $i\leftarrow \min\{s:a_s=a_j\}$, $t\leftarrow i-1$, $j\leftarrow i-1$, $S\leftarrow S\cup\{r_{i,a_i}\}$,\\
        \hspace*{3cm}\textbf{while} $t\geqslant 1$ \textbf{do}\\
        \hspace*{4cm}$l_t\leftarrow \min\{l_{t},l_i-1\}$, $t\leftarrow t-1$;\\
        \hspace*{2cm}\textbf{else} $i\leftarrow j$, $t\leftarrow i-1$, $j\leftarrow i-1$, $S\leftarrow S\cup\{e_{i,a_i}\}$,\\
        \hspace*{3cm}\textbf{while} $t\geqslant 1$ \textbf{do}\\
        \hspace*{4cm}$l_t\leftarrow \min\{l_{t},a_i\}$, $t\leftarrow t-1$;\\
        \hspace*{1cm}\textbf{else} $j\leftarrow j-1$.\\
        (3) Output $S_M\leftarrow S$.
    \end{alg}

    Note that Algorithm \ref{alg1} runs in time $O(m)$. For instance, if we run the above algorithm on input $CHS(3,3,3,4,5;1,1,2,2,3)$ and $M=(0,3,3,4,4)$, then we could get an edge subset $\{e_{5,4},r_{4,4},r_{2,3},e_{1,0}\}$ illustrated with a set of double lines in Fig. \ref{mono-chs}(b). On the other hand, there is a set of disjoint $M$-alternating cycles $\{C_{5,5},C_{4,4},C_{2,3},C_{1,1}\}$ illustrated with a set of solid cycles in Fig. \ref{mono-chs}(b).

\section{CHS with one turning}
    We now investigate another CHS, called \emph{CHS with one turning}. It can be obtained as follows from two monotonic CHS's, say $CHS(\{k_s;h_s\}_{s=1}^{m})$ and $CHS(\{k'_t;h'_t\}_{t=1}^{m'})$ with $k_m-h_m=k'_{m'}-h'_{m'}$. First place the two ones in left-monotonic way, then invert the second one upside down, and at last paste the $m$th row of the first one and the $m'$th row of the second one, see Fig. \ref{chs-tur}. The pasted row is called \emph{turning row}. We denote the CHS by $CHS(k_1,k_2,\ldots,k_m;h_1,h_2,\ldots,h_m|k'_1,k'_2,\ldots,k'_{m'};h'_1,h'_2,\ldots,h'_{m'})$, or briefly $CHS(\{k_s;$\\$h_s\}_{s=1}^{m}|\{k'_t;h'_t\}_{t=1}^{m'})$. What's more, its labels of hexagons and edges follow the corresponding two monotonic ones with the second one adding an apostrophe. Note that the hexagons and their edges in the turning row have two labels, such as $C_{m,i}=C'_{m',i-h_m+h'_{m'}}$.

    \begin{figure}[htbp]
        \centering
        \includegraphics[height=2.8in]{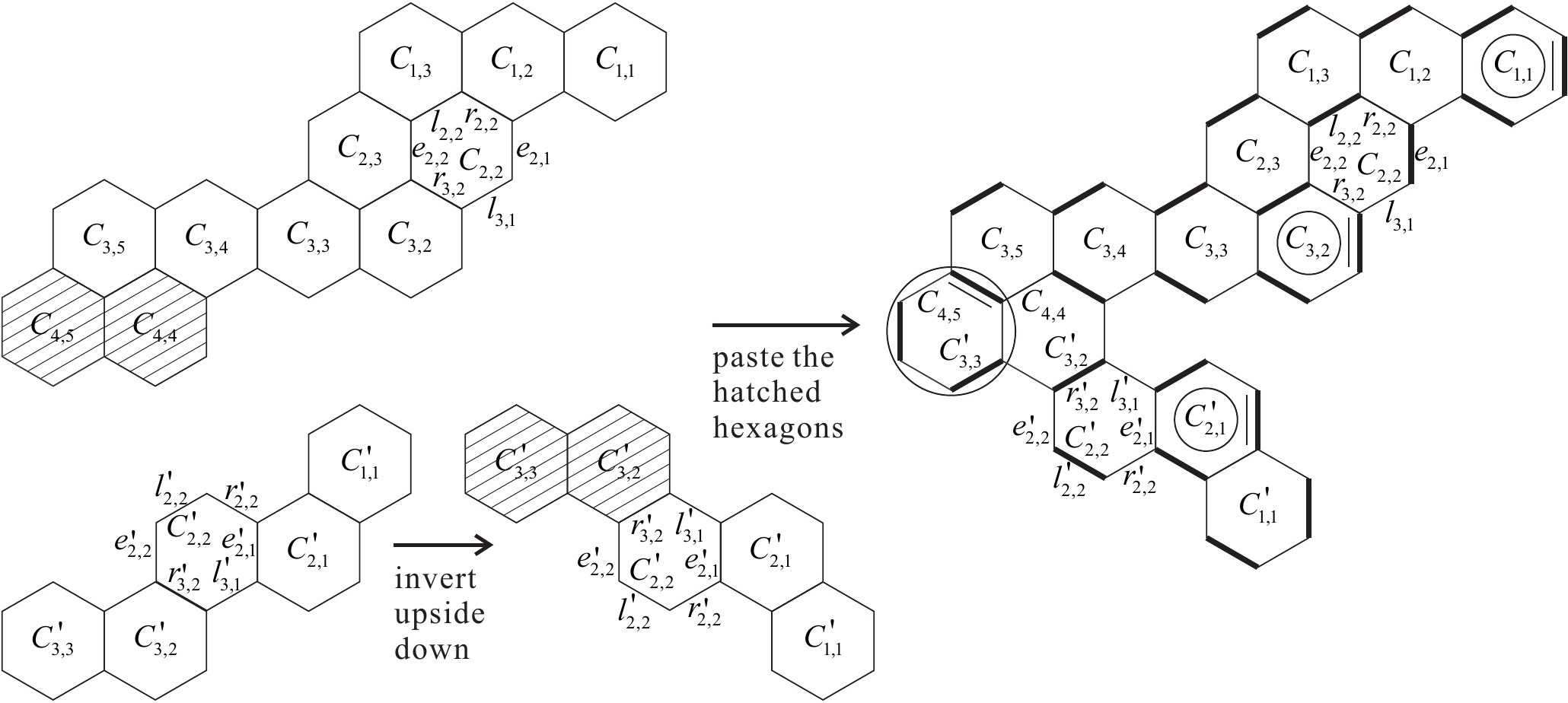}
        \caption{The way of obtaining $CHS(\{k_s;h_s\}_{s=1}^{m}|\{k'_t;h'_t\}_{t=1}^{m'})$ from $CHS(\{k_s;h_s\}_{s=1}^{m})$ and $CHS(\{k'_t;h'_t\}_{t=1}^{m'})$.}
        \label{chs-tur}
    \end{figure}

    Note that a CHS with one turning can be placed and represented in other ways. In particular, $CHS(\{k_s;h_s\}_{s=1}^{m}|\{k'_t;h'_t\}_{t=1}^{1})$ is a monotonic CHS. And we illustrate some examples of CHS with one turning in Figs. \ref{eg-chs-tur}(a-d). From now on suppose $m,m'\geqslant 2$.

    \begin{figure}[htbp]
        \centering
        \includegraphics[height=1.9in]{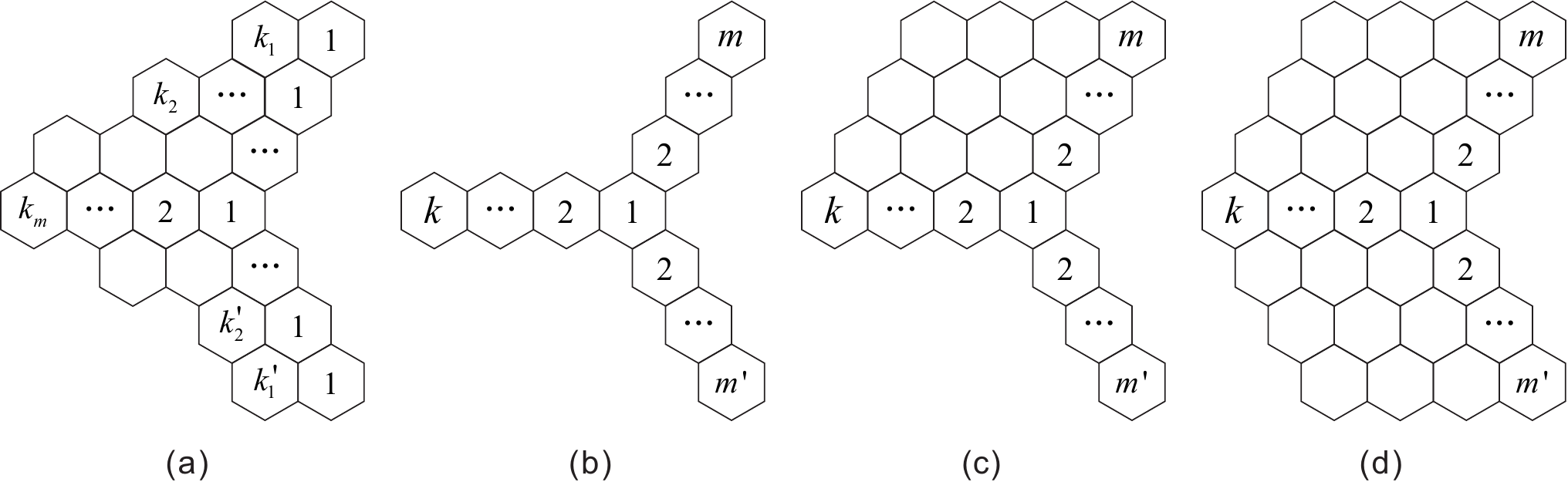}
        \caption{(a) $CHS(\{k_s;1\}_{s=1}^{m}|\{k'_t;1\}_{t=1}^{m'})$, (b) $CHS(1,\ldots,1,k;1,\ldots,1|1,\ldots,1,k;1,\ldots,1)$, (c) $CHS(k,\ldots,k,k;1,1,\ldots,1|1,\ldots,1,k;1,1,\ldots,1)$, and (d) $CHS(\{k;1\}_{s=1}^{m}|\{k;1\}_{t=1}^{m'})$.}
        \label{eg-chs-tur}
    \end{figure}

    In order to derive forcing polynomial of CHS's with one turning, we do some preliminaries first.

    \begin{lem}{\em\cite{dingli2}}
        \label{ding2}
        There is a bijection $g'$ between all perfect matchings $M$ of $CHS(\{k_s;$\\$h_s\}_{s=1}^{m}|\{k'_t;h'_t\}_{t=1}^{m'})$ and all binary non-decreasing sequences $((a_1,a_2,\ldots,a_m),(a'_1,a'_2,\ldots,$\\$a'_{m'}))$ with $a_m-h_m=a'_{m'}-h'_{m'},$ $a_i\in \{h_i-1,h_i,\ldots,k_i\},$ $a'_j\in \{h'_j-1,h'_j,\ldots,k'_j\}$ for $i=1,2,\ldots,m,$ $j=1,2,\ldots,m',$ such that $g'(M)=((a_1,a_2,\ldots,a_m),(a'_1,a'_2,\ldots,a'_{m'}))$ if $e_{i,a_i},e'_{j,a'_j}\in M$ for $i=1,2,\ldots,m,$ $j=1,2,\ldots,m'.$
    \end{lem}

    Obviously, $g'(M)=(g(M\cap E(CHS(\{k_s;h_s\}_{s=1}^{m}))),g(M\cap E(CHS(\{k'_t;h'_t\}_{t=1}^{m'}))))$. From the above lemma, we can use a sequence to express each perfect matching of a CHS with one turning for convenience. As an example, the perfect matching of $CHS(3,3,5,5;$\\$1,2,2,4|1,2,3;1,1,2)$ illustrated with a set of bold lines in Fig. \ref{chs-tur} can be expressed by ((0,1,1,5),(0,0,3)).

    Now we can give a recurrence relation of the forcing polynomial for $CHS(\{k_s;h_s\}_{s=1}^{m}|$\\$\{k'_t;h'_t\}_{t=1}^{m'})$. For convenience, we make a convention that
    \begin{align*}
        \mathcal{M}(\{k_s;h_s\}_{s=1}^{m}|\{k'_t;h'_t\}_{t=1}^{m'})&=\mathcal{M}(CHS(\{k_s;h_s\}_{s=1}^{m}|\{k'_t;h'_t\}_{t=1}^{m'})),\\
        F(\{k_s;h_s\}_{s=1}^{m}|\{k'_t;h'_t\}_{t=1}^{m'})&=F(CHS(\{k_s;h_s\}_{s=1}^{m}|\{k'_t;h'_t\}_{t=1}^{m'}),x).
    \end{align*}
    According to which vertical edge in the turning row that belongs to perfect matching, we divide $\mathcal{M}(\{k_s;h_s\}_{s=1}^{m}|\{k'_t;h'_t\}_{t=1}^{m'})$ in $k_m-h_m+2$ subsets:
    \begin{align*}
        \mathcal{M}_i(\{k_s;h_s\}_{s=1}^{m}|\{k'_t;h'_t\}_{t=1}^{m'})=\{M\in \mathcal{M}(\{k_s;h_s\}_{s=1}^{m}|\{k'_t;h'_t\}_{t=1}^{m'}): e_{m,i}\in M\}
    \end{align*}
    for $i=h_m-1,h_m,\ldots,k_m$. By Eq. (\ref{equ1}), we have
    \begin{align}
        \label{tur0}
        F(\{k_s;h_s\}_{s=1}^{m}|\{k'_t;h'_t\}_{t=1}^{m'})=&\sum_{i=h_m-1}^{k_m}\sum \limits_{ M\in {\mathcal{M}_i(\{k_s;h_s\}_{s=1}^{m}|\{k'_t;h'_t\}_{t=1}^{m'})}}{ x^{ f(CHS(\{k_s;h_s\}_{s=1}^{m}|\{k'_t;h'_t\}_{t=1}^{m'}),M) } } \nonumber \\
        :=&\sum_{i=h_m-1}^{k_m}F_i(\{k_s;h_s\}_{s=1}^{m}|\{k'_t;h'_t\}_{t=1}^{m'}).
    \end{align}

    \begin{thm}
        \label{chs-tur-forc}
        The forcing polynomial of $CHS(\{k_s;h_s\}_{s=1}^{m}|\{k'_t;h'_t\}_{t=1}^{m'})$ $(k_{i+1}\geqslant k_{i},$ $h_{i+1}\geqslant h_{i},$ $k_j\geqslant h_j$ for $i=1,2,\ldots,m-1,$ $j=1,2,\ldots,m,$ $k'_{u+1}\geqslant k'_{u},$ $h'_{u+1}\geqslant h'_{u},$ $k'_v\geqslant h'_v$ for $u=1,2,\ldots,m'-1,$ $v=1,2,\ldots,m',$ $m,m'\geqslant 2)$ has the following form$:$

        \noindent$(1)$ if $k_{m-1}+k'_{m'-1}<k_m+k'_{m'},$ then
        \begin{align}
            \label{chs-tur-eq1}
            &F(\{k_s;h_s\}_{s=1}^{m}|\{k'_t;h'_t\}_{t=1}^{m'})\nonumber \\
            =&\sum_{i=h_m-1}^{k_m-1}F(\{\min \{k_s,i\};h_s\}_{s=1}^{m-1})F(\{\min\{k'_t,i-h_m+h'_{m'}\};h'_t\}_{t=1}^{m'-1})x\nonumber\\
            &+\sum_{j=p(k_{m})}^{m}\sum_{i=q}^{m'}F(\{\min \{k_s,k_m-1\};h_s\}_{s=1}^{j-1})F(\{\min\{k'_t,k'_{m'}-1\};h'_t\}_{t=1}^{i-1})x;
        \end{align}
        $(2)$ if $k_{m-1}+k'_{m'-1}=k_m+k'_{m'},$ and the maximal zigzag hexagonal chain $\mathcal{Z}$ starting from $C_{m,k_m}$ (see Figs. \ref{eg-mono-chs} (c,d)) contains $n(\geqslant 2)$ hexagons$,$ namely $\mathcal{Z}=C_{m,k_m}C_{m-1,k_m}C_{m-1,k_m-1}$\\$C_{m-2,k_m-1}C_{m-2,k_m-2}\cdots C_{m-\lfloor\frac{n}{2}\rfloor,k_m-\lfloor\frac{n-1}{2}\rfloor},$ then
        \begin{align}
            \label{chs-tur-eq2}
            &F(\{k_s;h_s\}_{s=1}^{m}|\{k'_t;h'_t\}_{t=1}^{m'})\nonumber \\
            =&\sum_{i=h_m-1}^{k_m-1}F(\{\min \{k_s,i\};h_s\}_{s=1}^{m-1})F(\{\min\{k'_t,i-h_m+h'_{m'}\};h'_t\}_{t=1}^{m'-1})x\nonumber\\
            +&\sum_{j=p(k_m)}^{m-1}F(\{\min \{k_s,k_m-1\};h_s\}_{s=1}^{j-1})F(\{k'_t;h'_t\}_{t=1}^{m'-1})x\nonumber\\
            +&\sum_{j=q}^{m'-1}F(\{k_s;h_s\}_{s=1}^{m-1})F(\{\min \{k'_t,k'_{m'}-1\};h'_t\}_{t=1}^{j-1})x\nonumber\\
            -&\sum_{i=p(k_m)}^{m-1}\sum_{j=q}^{m'-1}F(\{\min\{k_s,k_m-1\};h_s\}_{s=1}^{i-1})F(\{\min \{k'_t,k'_{m'}-1\};h'_t\}_{t=1}^{j-1})x^2\nonumber \\
            +&\sum_{i=h_{m-1}-1}^{k_{m}-2}F(\{\min \{k_s,i\};h_s\}_{s=1}^{m-2})F(\{\min\{k'_t,k'_{m'}-1\};h'_t\}_{t=1}^{m'-1})x^2\nonumber\\
            +&\sum_{w=1}^{\lfloor\frac{n}{2}\rfloor-1}\sum_{i=p(k_m-w)}^{m-w-1}\sum_{j=h'_{{m'}-1}-1}^{k'_{m'}-1}F(\{\min \{k_s,k_m-w-1\};h_s\}_{s=1}^{i-1})F(\{\min\{k'_t,j\};h'_t\}_{t=1}^{m'-2})x^{w+2}\nonumber\\
            +&\sum_{w=1}^{\lfloor\frac{n}{2}\rfloor-1}\sum_{i=h_{m-w-1}-1}^{k_m-w-2}F(\{\min \{k_s,i\};h_s\}_{s=1}^{m-w-2})F(\{\min\{k'_t,k'_{m'}-1\};h'_t\}_{t=1}^{m'-1})x^{w+2}+{\xi }_{n},
        \end{align}
        where $p(z)=\min\{p:C_{p,z}\text{ exists}\},$ $q=\min\{q:C'_{q,k'_{m'}}\text{ exists}\},$ and ${\xi }_{n}=$
        \begin{align*}
            &\left\{ \begin{aligned}
            &\sum_{j=h'_{m'-1}-1}^{k'_{m'}-1}F(\{\min \{k_s,k_m-\frac{n}{2}\};h_s\}_{s=1}^{m-\frac{n}{2}-1})F(\{\min\{k'_t,j\};h'_t\}_{t=1}^{m'-2})x^{\frac{n}{2}+1} ~~~~\text{ if $n$ is even},\\
            &F(\{k_s;h_s\}_{s=1}^{m-\frac{n+1}{2}})F(\{\min\{k'_t,k'_{m'}-1\};h'_t\}_{t=1}^{m'-1})x^{\frac{n+1}{2}}~~~~~~~~~~~~~~~~~~~~~~~~~~~~~\text{if $n$ is odd}.\\
            \end{aligned} \right.
        \end{align*}
    \end{thm}
    \begin{proof}
        Given $i\in \{h_m-1,h_m,\ldots,k_m-1\}$ and $M\in \mathcal{M}_i(\{k_s;h_s\}_{s=1}^{m}|\{k'_t;h'_t\}_{t=1}^{m'})$. On the one hand $e_{m,i}$ belongs to $M$-alternating hexagon $C_{m,i+1}$, and on the other hand in $CHS(\{k_s;h_s\}_{s=1}^{m}|\{k'_t;h'_t\}_{t=1}^{m'})-V(e_{m,i})$ the lowermost (resp. uppermost) vertex of $C_{m,i+1}$ must be covered by the edge $r_{m+1,i+1}$ (resp. $l_{m,i+1}$) in $M$ by Lemma \ref{ding2}. It follows that $e_{m,i}$ belongs to some minimum forcing set of $M$ by Lemma \ref{mini-forc}. Furthermore, it is observed from Fig. \ref{chs-tur-proof}(a) that
        \begin{align*}
            &CHS(\{k_s;h_s\}_{s=1}^{m}|\{k'_t;h'_t\}_{t=1}^{m'})\circleddash V(e_{m,i})\\
            =&CHS(\{\min \{k_s,i\};h_s\}_{s=1}^{m-1})\cup CHS(\{\min\{k'_t,i-h_m+h'_{m'}\};h'_t\}_{t=1}^{m'-1}).
        \end{align*}
        By Lemma \ref{forc-subs-calc}, we have
        \begin{align}
            \label{tur1}
            &F_i(\{k_s;h_s\}_{s=1}^{m}|\{k'_t;h'_t\}_{t=1}^{m'})\nonumber \\
            =&\sum _{ M\in {\mathcal{M}_i(\{k_s;h_s\}_{s=1}^{m}|\{k'_t;h'_t\}_{t=1}^{m'})}}[{ x^{ f(CHS(\{\min \{k_s,i\};h_s\}_{s=1}^{m-1}),M\cap E(CHS(\{\min \{k_s,i\};h_s\}_{s=1}^{m-1})))}}\nonumber\\
            &\cdot { x^{ f(CHS(\{\min\{k'_t,i-h_m+h'_{m'}\};h'_t\}_{t=1}^{m'-1}),M\cap E(CHS(\{\min\{k'_t,i-h_m+h'_{m'}\};h'_t\}_{t=1}^{m'-1})))}}\cdot x]\nonumber \\
            =&\sum _{ M\in {\mathcal{M}(\{\min \{k_s,i\};h_s\}_{s=1}^{m-1})}}{x^{f(CHS(\{\min \{k_s,i\};h_s\}_{s=1}^{m-1}),M)}} \nonumber\\
            &\cdot \sum _{ M'\in {\mathcal{M}(\{\min\{k'_t,i-h_m+h'_{m'}\};h'_t\}_{t=1}^{m'-1})}}{x^{f(CHS(\{\min\{k'_t,i-h_m+h'_{m'}\};h'_t\}_{t=1}^{m'-1}),M')}} \cdot x \nonumber\\
            =&F(\{\min \{k_s,i\};h_s\}_{s=1}^{m-1})F(\{\min\{k'_t,i-h_m+h'_{m'}\};h'_t\}_{t=1}^{m'-1})x.
        \end{align}

        \begin{figure}[htbp]
            \centering
            \includegraphics[height=1.75in]{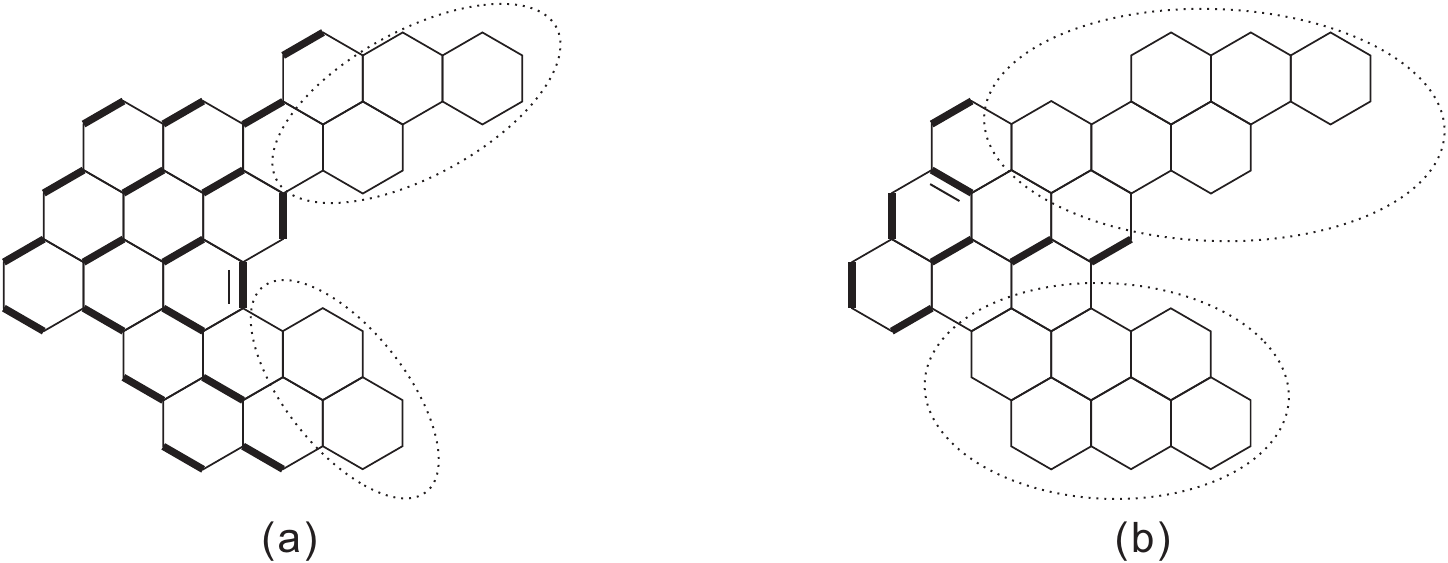}
            \caption{(a) $CHS(\{k_s;h_s\}_{s=1}^{m}|\{k'_t;h'_t\}_{t=1}^{m'})\circleddash V(e_{m,i})$, and (b) $CHS(\{k_s;h_s\}_{s=1}^{m}|\{k'_t;h'_t\}_{t=1}^{m'})\circleddash V(r_{j,k_m})$.}
            \label{chs-tur-proof}
        \end{figure}

        In the remaining part, we calculate $F_{k_m}(\{k_s;h_s\}_{s=1}^{m}|\{k'_t;h'_t\}_{t=1}^{m'})$ according to different values of $k_m$, $k_{m-1}$, $k'_{m'}$ and $k'_{{m'}-1}$.

        \textbf{Case 1.} $k_{m-1}\leqslant k_m$ and $k'_{m'-1}\leqslant k'_{m'}-1$. Then $p(k_m)\leqslant m$ and $q=m'$. Given $M\in \mathcal{M}_{k_m}(\{k_s;h_s\}_{s=1}^{m}|\{k'_t;h'_t\}_{t=1}^{m'})$. Let $j=\min\{j:e_{j,k_m}\in M\}$. Then $p(k_m)\leqslant j\leqslant m$ and $r_{j,k_m}\in M$. On the one hand $r_{j,k_m}$ belongs to $M$-alternating hexagon $C_{j,k_m}$, and on the other hand in $CHS(\{k_s;h_s\}_{s=1}^{m}|\{k'_t;h'_t\}_{t=1}^{m'})-V(r_{j,k_m})$ the leftmost two vertices of $C_{j,k_m}$ must be matched with each other in $M$, and the lowermost vertex of $C_{j,k_m}$ must be covered by the edge $l_{j+1,k_m-1}$ in $M$ by Lemma \ref{ding2}. It follows that $r_{j,k_m}$ belongs to some minimum forcing set of $M$ by Lemma \ref{mini-forc}. Furthermore, it is observed from Fig. \ref{chs-tur-proof}(b) that
        \begin{align*}
            &CHS(\{k_s;h_s\}_{s=1}^{m}|\{k'_t;h'_t\}_{t=1}^{m'})\circleddash V(r_{j,k_m})\\
            =&CHS(\{\min \{k_s,k_m-1\};h_s\}_{s=1}^{j-1})\cup CHS(\{k'_t;h'_t\}_{t=1}^{m'-1}).
        \end{align*}
        Similar to the calculation of Eq. (\ref{tur1}), by Lemma \ref{forc-subs-calc} we have
        \begin{align}
            \label{tur2}
            &F_{k_m}(\{k_s;h_s\}_{s=1}^{m}|\{k'_t;h'_t\}_{t=1}^{m'})=\sum_{j=p(k_{m})}^{m}F(\{\min \{k_s,k_m-1\};h_s\}_{s=1}^{j-1})F(\{k'_t;h'_t\}_{t=1}^{m'-1})x.
        \end{align}
        Substituting Eqs. (\ref{tur1},\ref{tur2}) into Eq. (\ref{tur0}), we immediately obtain Eq. (\ref{chs-tur-eq1}) in this case.

        \textbf{Case 2.} $k_{m-1}\leqslant k_m-1$ and $k'_{m'-1}=k'_{m'}$. Then $p(k_m)=m$ and $q\leqslant m'-1$. Given $M\in \mathcal{M}_{k_m}(\{k_s;h_s\}_{s=1}^{m}|\{k'_t;h'_t\}_{t=1}^{m'})$. Let $i=\min\{i:e'_{i,k'_{m'}}\in M\}$. Then $q\leqslant i\leqslant m'$ and $r'_{i,k'_{m'}}\in M$. By a similar argument to Case 1, we can derive
        \begin{align}
            \label{tur3}
            &F_{k_m}(\{k_s;h_s\}_{s=1}^{m}|\{k'_t;h'_t\}_{t=1}^{m'})=\sum_{i=q}^{m'}F(\{k_s;h_s\}_{s=1}^{m-1})F(\{\min\{k'_t,k'_{m'}-1\};h'_t\}_{t=1}^{i-1})x.
        \end{align}
        Substituting Eqs. (\ref{tur1},\ref{tur3}) into Eq. (\ref{tur0}), we immediately obtain Eq. (\ref{chs-tur-eq1}) in this case.

        \textbf{Case 3.} $k_{m-1}=k_m$ and $k'_{m'-1}=k'_{m'}$. Then $p(k_m)\leqslant m-1$ and $q\leqslant m'-1$. For $1\leqslant w\leqslant m-1$, denote
        \begin{align*}
            P(m-w)=\sum_{\begin{subarray}{c}M\in {\mathcal{M}_{k_m}(\{k_s;h_s\}_{s=1}^{m}|\{k'_t;h'_t\}_{t=1}^{m'})} \\ e'_{m'-1,k'_{m'}}\notin M, e_{m-i,k_m-i}\in M \text{ for }i=1,2,\ldots,w\end{subarray}}{ x^{ f(CHS(\{k_s;h_s\}_{s=1}^{m}|\{k'_t;h'_t\}_{t=1}^{m'}),M) } }.
        \end{align*}
        By subdividing $\mathcal{M}_{k_m}(\{k_s;h_s\}_{s=1}^{m}|\{k'_t;h'_t\}_{t=1}^{m'})$ in more subsets, we have
        \begin{align}
            \label{tur14}
            F_{k_m}(\{k_s;h_s\}_{s=1}^{m}|\{k'_t;h'_t\}_{t=1}^{m'})=&\sum_{\begin{subarray}{c}M\in {\mathcal{M}_{k_m}(\{k_s;h_s\}_{s=1}^{m}|\{k'_t;h'_t\}_{t=1}^{m'})} \\ e_{m-1,k_m}\in M\end{subarray}}{ x^{ f(CHS(\{k_s;h_s\}_{s=1}^{m}|\{k'_t;h'_t\}_{t=1}^{m'}),M) } }\nonumber\\
            &+\sum_{\begin{subarray}{c}M\in {\mathcal{M}_{k_m}(\{k_s;h_s\}_{s=1}^{m}|\{k'_t;h'_t\}_{t=1}^{m'})} \\ e'_{m'-1,k'_{m'}}\in M\end{subarray}}{ x^{ f(CHS(\{k_s;h_s\}_{s=1}^{m}|\{k'_t;h'_t\}_{t=1}^{m'}),M) } }\nonumber\\
            &-\sum_{\begin{subarray}{c}M\in {\mathcal{M}_{k_m}(\{k_s;h_s\}_{s=1}^{m}|\{k'_t;h'_t\}_{t=1}^{m'})} \\ e_{m-1,k_m},e'_{m'-1,k'_{m'}}\in M\end{subarray}}{ x^{ f(CHS(\{k_s;h_s\}_{s=1}^{m}|\{k'_t;h'_t\}_{t=1}^{m'}),M) } }\nonumber\\
            &+\sum_{\begin{subarray}{c}M\in {\mathcal{M}_{k_m}(\{k_s;h_s\}_{s=1}^{m}|\{k'_t;h'_t\}_{t=1}^{m'})} \\ e_{m-1,k_m},e_{m-1,k_m-1},e'_{m'-1,k'_{m'}}\notin M\end{subarray}}{ x^{ f(CHS(\{k_s;h_s\}_{s=1}^{m}|\{k'_t;h'_t\}_{t=1}^{m'}),M) } }\nonumber\\
            &+\sum_{\begin{subarray}{c}M\in {\mathcal{M}_{k_m}(\{k_s;h_s\}_{s=1}^{m}|\{k'_t;h'_t\}_{t=1}^{m'})} \\ e_{m-1,k_m-1}\in M,e'_{m'-1,k'_{m'}}\notin M\end{subarray}}{ x^{ f(CHS(\{k_s;h_s\}_{s=1}^{m}|\{k'_t;h'_t\}_{t=1}^{m'}),M) } }\nonumber\\
            :=&P_1+P_2-P_3+P_4+P(m-1).
        \end{align}

        Given $M\in \mathcal{M}_{k_m}(\{k_s;h_s\}_{s=1}^{m}|\{k'_t;h'_t\}_{t=1}^{m'})$ with $e_{m-1,k_m}\in M$. Let $j=\min\{j:e_{j,k_m}\in M\}$. Then $p(k_m)\leqslant j\leqslant m-1$ and $r_{j,k_m}\in M$. Note that $C_{j,k_m}$ is an $M$-alternating hexagon containing $r_{j,k_m}$. By Lemma \ref{ding2} and a similar argument to Case 1, we can derive that $r_{j,k_m}$ belongs to some minimum forcing set of $M$. Furthermore, it is observed from Fig. \ref{chs-tur-proof1}(a) that
        \begin{align*}
            &CHS(\{k_s;h_s\}_{s=1}^{m}|\{k'_t;h'_t\}_{t=1}^{m'})\circleddash V(r_{j,k_m})\\
            =&CHS(\{\min \{k_s,k_m-1\};h_s\}_{s=1}^{j-1})\cup CHS(\{k'_t;h'_t\}_{t=1}^{m'-1}).
        \end{align*}
        Similar to the calculation of Eq. (\ref{tur1}), by Lemma \ref{forc-subs-calc} we have
        \begin{align}
            \label{tur4}
            P_1=\sum_{j=p(k_m)}^{m-1}F(\{\min \{k_s,k_m-1\};h_s\}_{s=1}^{j-1})F(\{k'_t;h'_t\}_{t=1}^{m'-1})x.
        \end{align}

        \begin{figure}[htbp]
            \centering
            \includegraphics[height=1.75in]{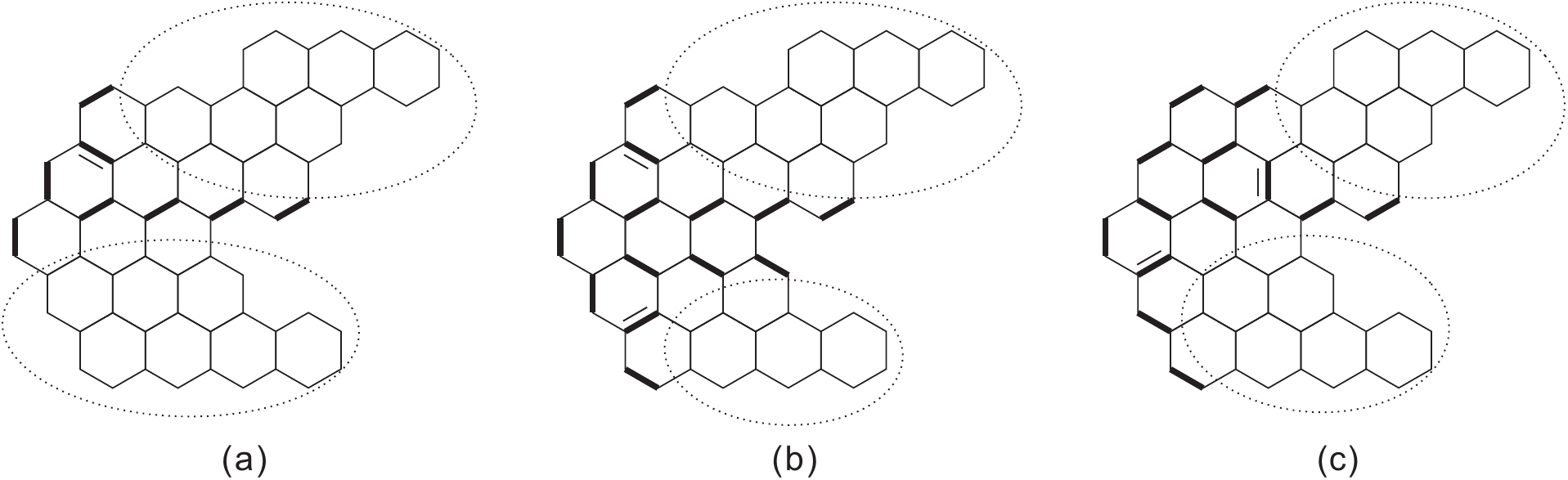}
            \caption{Illustration of calculations for $P_1$, $P_3$ and $P_4$.}
            \label{chs-tur-proof1}
        \end{figure}

        Given $M\in \mathcal{M}_{k_m}(\{k_s;h_s\}_{s=1}^{m}|\{k'_t;h'_t\}_{t=1}^{m'})$ with $e'_{m'-1,k'_{m'}}\in M$. By a similar argument to the calculation of $P_1$, we can derive that
        \begin{align}
            \label{tur11}
            P_2=\sum_{j=q}^{m'-1}F(\{k_s;h_s\}_{s=1}^{m-1})F(\{\min \{k'_t,k'_{m'}-1\};h'_t\}_{t=1}^{j-1})x.
        \end{align}

        Given $M\in \mathcal{M}_{k_m}(\{k_s;h_s\}_{s=1}^{m}|\{k'_t;h'_t\}_{t=1}^{m'})$ with $e_{m-1,k_m},e'_{m'-1,k'_{m'}}\in M$. Let $i=\min\{i:e_{i,k_m}\in M\}$ and $j=\min\{j:e'_{j,k'_{m'}}\in M\}$. Then $p(k_m)\leqslant i\leqslant m-1$, $q\leqslant j\leqslant m'-1$, and $r_{i,k_m},r'_{j,k'_{m'}}\in M$. On the one hand $r_{i,k_m}$ belongs to $M$-alternating hexagon $C_{i,k_m}$, $r'_{j,k'_{m'}}$ belongs to $M$-alternating hexagon $C'_{j,k'_{m'}}$, and the two hexagons are disjoint. And on the other hand in $CHS(\{k_s;h_s\}_{s=1}^{m}|\{k'_t;h'_t\}_{t=1}^{m'})-V(\{r_{i,k_m},r'_{j,k'_{m'}}\})$, the leftmost two vertices of $C_{i,k_m}$ (resp. $C'_{j,k'_{m'}}$) must be matched with each other in $M$, and the lowermost (resp. uppermost) vertex of $C_{i,k_m}$ (resp. $C'_{j,k'_{m'}}$) must be covered by the edge $l_{i+1,k_m-1}$ (resp. $l'_{j+1,k'_{m'}-1}$) in $M$ by Lemma \ref{ding2}. It follows that $\{r_{i,k_m},r'_{j,k'_{m'}}\}$ is contained in some minimum forcing set of $M$ by Lemma \ref{mini-forc}. Furthermore, it is observed from Fig. \ref{chs-tur-proof1}(b) that
        \begin{align*}
            &CHS(\{k_s;h_s\}_{s=1}^{m}|\{k'_t;h'_t\}_{t=1}^{m'})\circleddash V(\{r_{i,k_m},r'_{j,k'_{m'}}\})\\
            =&CHS(\{\min \{k_s,k_m-1\};h_s\}_{s=1}^{i-1})\cup CHS(\{\min\{k'_t,k'_{m'}-1\};h'_t\}_{t=1}^{j-1}).
        \end{align*}
        Similar to the calculation of Eq. (\ref{tur1}), by Lemma \ref{forc-subs-calc} we have
        \begin{align}
            \label{tur12}
            P_3=&\sum_{i=p(k_m)}^{m-1}\sum_{j=q}^{m'-1}F(\{\min\{k_s,k_m-1\};h_s\}_{s=1}^{i-1})F(\{\min \{k'_t,k'_{m'}-1\};h'_t\}_{t=1}^{j-1})x^2.
        \end{align}

        Given $M\in \mathcal{M}_{k_m}(\{k_s;h_s\}_{s=1}^{m}|\{k'_t;h'_t\}_{t=1}^{m'})$ with $e_{m-1,i}\in M$ and $e'_{m'-1,k'_{m'}}\notin M$ for $h_{m-1}-1\leqslant i\leqslant k_m-2$. Then $r'_{m',k'_{m'}}\in M$. Note that $C_{m-1,i+1}$ and $C'_{m',k'_{m'}}$ are disjoint $M$-alternating hexagons containing $e_{m-1,i}$ and $r'_{m',k'_{m'}}$, respectively. By Lemma \ref{ding2} and a similar argument to the calculation of $P_3$, we can derive that $\{e_{m-1,i},r'_{m',k'_{m'}}\}$ is contained in some minimum forcing set of $M$. Furthermore, it is observed from Fig. \ref{chs-tur-proof1}(c) that
        \begin{align*}
            &CHS(\{k_s;h_s\}_{s=1}^{m}|\{k'_t;h'_t\}_{t=1}^{m'})\circleddash V(\{e_{m-1,i},r'_{m',k'_{m'}}\})\\
            =&CHS(\{\min \{k_s,i\};h_s\}_{s=1}^{m-2})\cup CHS(\{\min\{k'_t,k'_{m'}-1\};h'_t\}_{t=1}^{m'-1}).
        \end{align*}
        Similar to the calculation of Eq. (\ref{tur1}), by Lemma \ref{forc-subs-calc} we have
        \begin{align}
            \label{tur6}
            P_4=\sum_{i=h_{m-1}-1}^{k_m-2}F(\{\min \{k_s,i\};h_s\}_{s=1}^{m-2})F(\{\min\{k'_t,k'_{m'}-1\};h'_t\}_{t=1}^{m'-1})x^2.
        \end{align}

        It remains to consider $P(m-1)$. Given $M\in \mathcal{M}_{k_m}(\{k_s;h_s\}_{s=1}^{m}|\{k'_t;h'_t\}_{t=1}^{m'})$ with $e_{m-1,k_m-1},e'_{m'-1,j}\in M$ for $h'_{m'-1}-1\leqslant j\leqslant k'_{m'}-1$. The value of $P(m-1)$ varies with different values of $n$, and we distinguish according to the following subcases.

        \textbf{Subcase 3.1.} $n=2$. Then the last hexagon of $\mathcal{Z}$ is $C_{m-1,k_m}$ and $C_{m,k_m-1}$ is nonexistent. Note that $C_{m-1,k_m}$ and $C'_{m'-1,j+1}$ are disjoint $M$-alternating hexagons containing $r_{m,k_m}$ and $e'_{m'-1,j}$, respectively. By Lemma \ref{ding2} and a similar argument to the calculation of $P_3$, we can derive that $\{r_{m,k_m},e'_{m'-1,j}\}$ is contained in some minimum forcing set of $M$. Furthermore, it is observed from Fig. \ref{chs-tur-proof2}(a) that
        \begin{align*}
            &CHS(\{k_s;h_s\}_{s=1}^{m}|\{k'_t;h'_t\}_{t=1}^{m'})\circleddash V(\{r_{m,k_m},e'_{m'-1,j}\})\\
            =&CHS(\{\min \{k_s,k_m-1\};h_s\}_{s=1}^{m-2})\cup CHS(\{\min\{k'_t,j\};h'_t\}_{t=1}^{m'-2}).
        \end{align*}
        Similar to the calculation of Eq. (\ref{tur1}), by Lemma \ref{forc-subs-calc} we have
        \begin{align*}
            P(m-1)=&\sum_{j=h'_{m'-1}-1}^{k'_{m'}-1}F(\{\min \{k_s,k_m-1\};h_s\}_{s=1}^{m-2})F(\{\min\{k'_t,j\};h'_t\}_{t=1}^{m'-2})x^2.
        \end{align*}

        \begin{figure}[htbp]
            \centering
            \includegraphics[height=1.9in]{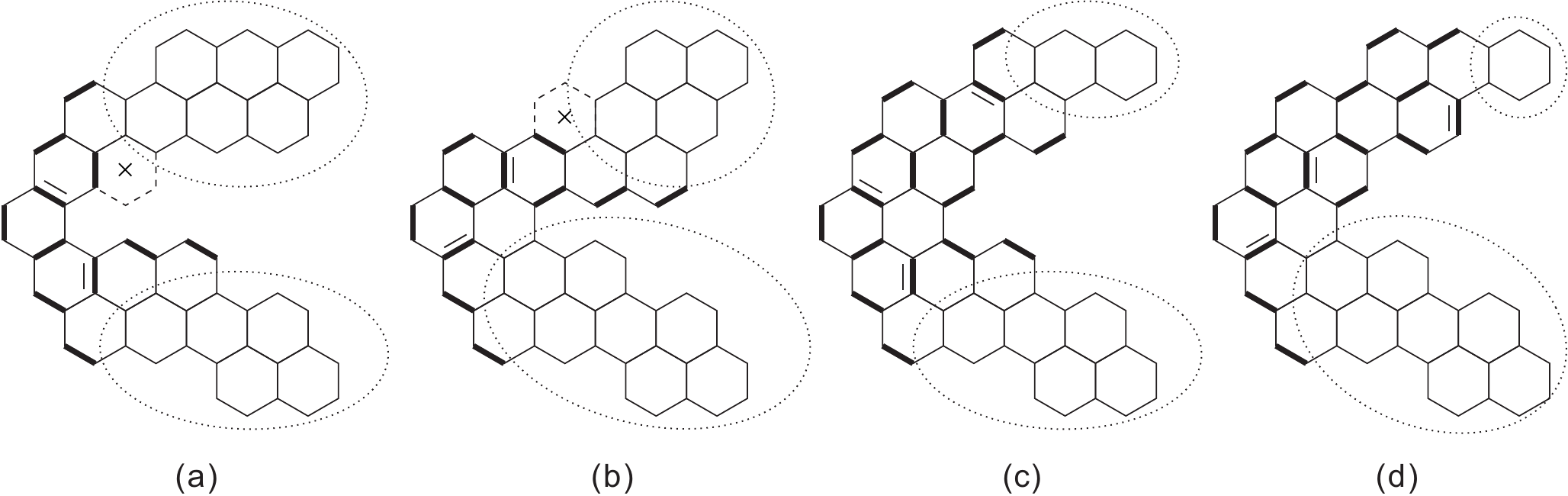}
            \caption{Illustration of calculation for $P(m-1)$.}
            \label{chs-tur-proof2}
        \end{figure}

        \textbf{Subcase 3.2.} $n=3$. Then the last hexagon of $\mathcal{Z}$ is $C_{m-1,k_m-1}$ and $C_{m-2,k_m}$ is nonexistent. Note that $C_{m-1,k_m-1}$ and $C'_{m',k'_{m'}}$ are disjoint $M$-alternating hexagons containing $e_{m-1,k_m-1}$ and $r'_{m',k'_{m'}}$, respectively. By Lemma \ref{ding2} and a similar argument to the calculation of $P_3$, we can derive that $\{e_{m-1,k_m-1},r'_{m',k'_{m'}}\}$ is contained in some minimum forcing set of $M$. Furthermore, it is observed from Fig. \ref{chs-tur-proof2}(b) that
        \begin{align*}
            &CHS(\{k_s;h_s\}_{s=1}^{m}|\{k'_t;h'_t\}_{t=1}^{m'})\circleddash V(\{e_{m-1,k_m-1},r'_{m',k'_{m'}}\})\\
            =&CHS(\{k_s;h_s\}_{s=1}^{m-2})\cup CHS(\{\min\{k'_t,k'_{m'}-1\};h'_t\}_{t=1}^{m'-1}).
        \end{align*}
        Similar to the calculation of Eq. (\ref{tur1}), by Lemma \ref{forc-subs-calc} we have
        \begin{align*}
            P(m-1)=&F(\{k_s;h_s\}_{s=1}^{m-2})F(\{\min\{k'_t,k'_{m'}-1\};h'_t\}_{t=1}^{m'-1})x^2.
        \end{align*}

        \textbf{Subcase 3.3.} $n\geqslant 4$. Then $m\geqslant 3$ and $p(k_m-1)\leqslant m-2$. By subdividing $\{M\in \mathcal{M}_{k_m}(\{k_s;h_s\}_{s=1}^{m}|\{k'_t;h'_t\}_{t=1}^{m'}):e_{m-1,k_m-1}\in M,e'_{m'-1,k'_{m'}}\notin M\}$ in more subsets, we have
        \begin{align}
            \label{tur13}
            P(m-1)=&\sum_{\begin{subarray}{c}M\in {\mathcal{M}_{k_m}(\{k_s;h_s\}_{s=1}^{m}|\{k'_t;h'_t\}_{t=1}^{m'})} \\ e_{m-1,k_m-1},e_{m-2,k_m-1},e'_{m'-1,j}\in M \text{ for }h'_{m'-1}-1 \leqslant j\leqslant k'_{m'}-1 \\ \end{subarray}}{ x^{ f(CHS(\{k_s;h_s\}_{s=1}^{m}|\{k'_t;h'_t\}_{t=1}^{m'}),M) } } \nonumber \\
            +&\sum_{\begin{subarray}{c}M\in {\mathcal{M}_{k_m}(\{k_s;h_s\}_{s=1}^{m}|\{k'_t;h'_t\}_{t=1}^{m'})} \\ e_{m-1,k_m-1}\in M,e'_{m'-1,k'_{m'}},e_{m-2,k_m},e_{m-2,k_m-1},e_{m-2,k_m-2}\notin M\end{subarray}}{ x^{ f(CHS(\{k_s;h_s\}_{s=1}^{m}|\{k'_t;h'_t\}_{t=1}^{m'}),M) } } \nonumber\\
            +&\sum_{\begin{subarray}{c}M\in {\mathcal{M}_{k_m}(\{k_s;h_s\}_{s=1}^{m}|\{k'_t;h'_t\}_{t=1}^{m'})} \\ e_{m-1,k_m-1},e_{m-2,k_m-2}\in M,e'_{m'-1,k'_{m'}}\notin M\end{subarray}}{ x^{ f(CHS(\{k_s;h_s\}_{s=1}^{m}|\{k'_t;h'_t\}_{t=1}^{m'}),M) } } \nonumber\\
            :=&P_5+P_6+P(m-2).
        \end{align}

        Given $M\in {\mathcal{M}_{k_m}(\{k_s;h_s\}_{s=1}^{m}|\{k'_t;h'_t\}_{t=1}^{m'})}$ with $e_{m-1,k_m-1},e_{m-2,k_m-1},e'_{m'-1,j}\in M$ for $h'_{m'-1}-1 \leqslant j\leqslant k'_{m'}-1$. Let $i=\min\{i:e_{i,k_m-1}\in M\}$. Then $p(k_m-1)\leqslant i\leqslant m-2$. Note that $C_{i,k_m-1}$, $C_{m-1,k_m}$ and $C'_{m'-1,j+1}$ are disjoint $M$-alternating hexagons containing $r_{i,k_m-1}$, $r_{m,k_m}$ and $e'_{m'-1,j}$, respectively. By Lemma \ref{ding2} and a similar argument to the calculation of $P_3$, we can derive that $\{r_{i,k_m-1},r_{m,k_m},e'_{m'-1,j}\}$ is contained in some minimum forcing set of $M$. Furthermore, it is observed from Fig. \ref{chs-tur-proof2}(c) that
        \begin{align*}
            &CHS(\{k_s;h_s\}_{s=1}^{m}|\{k'_t;h'_t\}_{t=1}^{m'})\circleddash V(\{r_{i,k_m-1},r_{m,k_m},e'_{m'-1,j}\})\\
            =&CHS(\{\min \{k_s,k_m-2\};h_s\}_{s=1}^{i-1})\cup CHS(\{\min\{k'_t,j\};h'_t\}_{t=1}^{m'-2}).
        \end{align*}
        Similar to the calculation of Eq. (\ref{tur1}), by Lemma \ref{forc-subs-calc} we have
        \begin{align}
            \label{tur9}
            P_5=&\sum_{i=p(k_m-1)}^{m-2}\sum_{j=h'_{{m'}-1}-1}^{k'_{m'}-1}F(\{\min \{k_s,k_m-2\};h_s\}_{s=1}^{i-1})F(\{\min\{k'_t,j\};h'_t\}_{t=1}^{m'-2})x^3.
        \end{align}

        Given $M\in {\mathcal{M}_{k_m}(\{k_s;h_s\}_{s=1}^{m}|\{k'_t;h'_t\}_{t=1}^{m'})}$ with $e_{m-1,k_m-1},e_{m-2,i}\in M$ and $e'_{m'-1,k'_{m'}}\notin M$ for $h_{m-2}-1\leqslant i\leqslant k_m-3$. Obviously there is a perfect matching which satisfies the condition if and only if $n\geqslant 5$. Note that $C_{m-2,i+1}$, $C_{m-1,k_m-1}$ and $C'_{m',k'_{m'}}$ are disjoint $M$-alternating hexagons containing $e_{m-2,i}$ $e_{m-1,k_m-1}$ and $r'_{m',k'_{m'}}$, respectively. By Lemma \ref{ding2} and a similar argument to the calculation of $P_3$, we can derive that $\{e_{m-2,i},e_{m-1,k_m-1},r'_{m',k'_{m'}}\}$ is contained in some minimum forcing set of $M$. Furthermore, it is observed from Fig. \ref{chs-tur-proof2}(d) that
        \begin{align*}
            &CHS(\{k_s;h_s\}_{s=1}^{m}|\{k'_t;h'_t\}_{t=1}^{m'})\circleddash V(\{e_{m-2,i},e_{m-1,k_m-1},r'_{m',k'_{m'}}\})\\
            =&CHS(\{\min \{k_s,i\};h_s\}_{s=1}^{m-3})\cup CHS(\{\min\{k'_t,k'_{m'}-1\};h'_t\}_{t=1}^{m'-1}).
        \end{align*}
        Similar to the calculation of Eq. (\ref{tur1}), by Lemma \ref{forc-subs-calc} we have
        \begin{align}
            \label{tur10}
            P_6=\sum_{i=h_{m-2}-1}^{k_m-3}F(\{\min \{k_s,i\};h_s\}_{s=1}^{m-3})F(\{\min\{k'_t,k'_{m'}-1\};h'_t\}_{t=1}^{m'-1})x^3.
        \end{align}
        Substituting Eqs. (\ref{tur9},\ref{tur10}) into Eq. (\ref{tur13}), we immediately obtain $P(m-1)$ in this case.

        Similar to the above three subcases, we have for $1\leqslant w\leqslant \lfloor\frac{n}{2}\rfloor$
        \begin{align*}
            &P(m-w)=\\
            &\left\{ \begin{aligned}
            &\sum_{j=h'_{m'-1}-1}^{k'_{m'}-1}F(\{\min \{k_s,k_m-w\};h_s\}_{s=1}^{m-w-1})F(\{\min\{k'_t,j\};h'_t\}_{t=1}^{m'-2})x^{w+1}~~~~~~~\text{ if } n=2w,\\
            &F(\{k_s;h_s\}_{s=1}^{m-w-1})F(\{\min\{k'_t,k'_{m'}-1\};h'_t\}_{t=1}^{m'-1})x^{w+1} ~~~~~~~~~~~~~~~~~~~~~~~~~\text{ if } n=2w+1,\\
            & \sum_{i=p(k_m-w)}^{m-w-1}\sum_{j=h'_{{m'}-1}-1}^{k'_{m'}-1}F(\{\min \{k_s,k_m-w-1\};h_s\}_{s=1}^{i-1})F(\{\min\{k'_t,j\};h'_t\}_{t=1}^{m'-2})x^{w+2}\\
            &+\sum_{i=h_{m-w-1}-1}^{k_m-w-2}F(\{\min \{k_s,i\};h_s\}_{s=1}^{m-w-2})F(\{\min\{k'_t,k'_{m'}-1\};h'_t\}_{t=1}^{m'-1})x^{w+2}\\
            &+P(m-w-1)~~~~~~~~~~~~~~~~~~~~~~~~~~~~~~~~~~~~~~~~~~~~~~~~~~~~~~~~~~~~~~~~~~~~~~~\text{ if }n\geqslant 2w+2.\\
            \end{aligned} \right.
        \end{align*}
        Substituting each $P(m-w)$ into $P(m-1)$ and substituting Eqs. (\ref{tur4}-\ref{tur6}) into Eq. (\ref{tur14}), we immediately obtain $F_{k_m}(\{k_s;h_s\}_{s=1}^{m}|\{k'_t;h'_t\}_{t=1}^{m'})$ in this case. Furthermore, substituting the result and Eq. (\ref{tur1}) into Eq. (\ref{tur0}), we immediately obtain Eq. (\ref{chs-tur-eq2}).
    \end{proof}

    Hansen and Zheng \cite{5} and Zhang and Li \cite{19} characterized hexagonal systems with forcing edges independently, which are CHS's with one turning while $h_s=h'_t=1$ for $s=1,2,\ldots,m$ and $t=1,2,\ldots,m'$ (see Fig. \ref{eg-chs-tur}(a)). The forcing polynomial can be derived from the above theorem. In the following, we give some forcing polynomials of particular CHS's with one turning, which are illustrated in Figs. \ref{eg-chs-tur}(b-d).

    \begin{eg}{\em\cite{zhao}}
        The forcing polynomial of $CHS(1,\ldots,1,k;1,1,\ldots,1|1,\ldots,1,k;1,1,$\\$\ldots,1)$ $(k\geqslant 1,~m,m'\geqslant 2)$ (see Fig. \ref{eg-chs-tur}(b)) has the following form$:$

        \noindent$(1)$ if $k=1,$ then
        \begin{align*}
            &F(\{1;1\}_{s=1}^{m}|\{1;1\}_{t=1}^{m'})\\
            =&x+\sum_{j=1}^{m-1}F(\{1;1\}_{t=1}^{m'-1})x+\sum_{j=1}^{m'-1}F(\{1;1\}_{s=1}^{m-1})x-\sum_{i=1}^{m-1}\sum_{j=1}^{m'-1}x^2+x^2=mm'x^2+x;
        \end{align*}
        $(2)$ if $k\geqslant 2,$ then
        \begin{align*}
            &F(1,\ldots,1,k;1,1,\ldots,1|1,\ldots,1,k;1,1,\ldots,1)\\
            =&\sum_{i=1}^{k-1}F(\{1;1\}_{s=1}^{m-1})F(\{1;1\}_{t=1}^{m'-1})x+x+F(\{1;1\}_{s=1}^{m-1})F(\{1;1\}_{t=1}^{m'-1})x=kmm'x^3+x.
        \end{align*}
    \end{eg}

    From the above conclusion, we know that the forcing spectrum of $CHS(\{k_s;h_s\}_{s=1}^{m}|\{k'_t;$\\$h'_t\}_{t=1}^{m'})$ is not always continuous, especially for hexagonal systems with forcing edges.

    \begin{eg}
        The forcing polynomial of $CHS(k,\ldots,k,k;1,1,\ldots,1|1,\ldots,1,k;1,1,\ldots,$\\$1)$ $(k,m,m'\geqslant 2)$ (see Fig. \ref{eg-chs-tur}(c)) is
        \begin{align*}
            &F(k,\ldots,k,k;1,1,\ldots,1|1,\ldots,1,k;1,1,\ldots,1)\\
            =&\sum_{i=1}^{k-1}F(\{i;1\}_{s=1}^{m-1})F(\{1;1\}_{t=1}^{m'-1})x+x+\sum_{j=1}^{m}F(\{k-1;1\}_{s=1}^{j-1})F(\{1;1\}_{t=1}^{m'-1})x\\
            =&m'F(M(k,m),x)x-m'x^2+x.
        \end{align*}
    \end{eg}

    \begin{eg}
        The forcing polynomial of $CHS(\{k;1\}_{s=1}^{m}|\{k;1\}_{t=1}^{m'})$ $(k,m,m'\geqslant 2)$ (see Fig. \ref{eg-chs-tur}(d)) is
        \begin{align*}
            &F(\{k;1\}_{s=1}^{m}|\{k;1\}_{t=1}^{m'})\\
            =&\sum_{i=0}^{k-1}F(\{i;1\}_{s=1}^{m-1})F(\{i;1\}_{t=1}^{m'-1})x+\sum_{j=1}^{m-1}F(\{k-1;1\}_{s=1}^{j-1})F(\{k;1\}_{t=1}^{m'-1})x\\
            &+\sum_{j=1}^{m'-1}F(\{k;1\}_{s=1}^{m-1})F(\{k-1;1\}_{t=1}^{j-1})x-\sum_{i=1}^{m-1}\sum_{j=1}^{m'-1}F(\{k-1;1\}_{s=1}^{i-1})F(\{k-1;1\}_{t=1}^{j-1})x^2 \\
            &+\sum_{i=0}^{k-2}F(\{i;1\}_{s=1}^{m-2})F(\{k-1;1\}_{t=1}^{m'-1})x^2\\
            &+\sum_{w=1}^{\min\{k-1,m-2\}}\sum_{i=1}^{m-w-1}\sum_{j=0}^{k-1}F(\{k-w-1;1\}_{s=1}^{i-1})F(\{j;1\}_{t=1}^{m'-2})x^{w+2}\\
            &+\sum_{w=1}^{\min\{k-1,m-2\}}\sum_{i=0}^{k-w-2}F(\{i;1\}_{s=1}^{m-w-2})F(\{k-1;1\}_{t=1}^{m'-1})x^{w+2}+F^{\ast},
        \end{align*}
        where
        \begin{align*}
            F^{\ast}=\left\{ \begin{aligned}
            &\sum_{j=0}^{k-1}F(\{j;1\}_{t=1}^{m'-2})x^{k+1} ~~~\text{ if } k<m,\\
            &F(\{k-1;1\}_{t=1}^{m'-1})x^m ~~~~~\text{ if }k\geqslant m.\\
            \end{aligned} \right.
        \end{align*}
    \end{eg}

\end{document}